\documentclass[12pt,reqno]{amsart}
\usepackage{graphicx}
\usepackage{hyperref}
\usepackage{psfrag}
\usepackage{pxfonts}
\usepackage{mathrsfs}
\usepackage{color}
\usepackage{amsmath,amssymb}

\numberwithin{equation}{section}

\newcommand{\Per}{{\rm Per}}

\theoremstyle{plain}
\newtheorem{maintheorem}{Theorem}

\newcommand{\R}{\mathbb{R}}
\newcommand{\N}{\mathbb{N}}
\newcommand{\Z}{\mathbb{Z}}

\newcommand{\T}{\mathbb{T}}

\newcommand{\dime}{\mathrm{dim}}

\newtheorem{theorem}{Theorem}[section]
\newtheorem{corolA1}{Corollary A\hspace{-0.1cm}}

\newtheorem{corollary}[theorem]{Corollary}
\newtheorem{proposition}[theorem]{Proposition}
\newtheorem{lemma}[theorem]{Lemma}
\newtheorem{definition}[theorem]{Definition}

\theoremstyle{remark}
\newtheorem{remark}[theorem]{Remark}

\begin{document}

\thanks{J. Santana C. Costa was supported by the Coordenação de Aperfeiçoamento de Pessoal de Nível Superior (CAPES) - Brasil, PRAPG, Process No. 88887.196178/2025-00. F. Micena was partially supported by the National Council for Scientific and Technological Development (CNPq), grant number 309185/2026-2, and by the FAPEMIG Research Network in Mathematics (Rede Mineira de Matemática), grant RED-00133-21.}

\author{J. Santana C. Costa}
\address{DEMAT-UFMA S\~{a}o Lu\'{i}s-MA, Brazil.}
\email{jsc.costa@ufma.br}

\author{F. Micena}
\address{
  IMC-UNIFEI Itajub\'{a}-MG, Brazil.}
\email{fpmicena82@unifei.edu.br}

\renewcommand{\subjclassname}{\textup{2000} Mathematics Subject Classification}

\date{\today}

\setcounter{tocdepth}{2}
\title{Aspects of the Liv\v{s}ic Theory for Anosov Endomorphisms} %Liv\v{s}ic--Sinai's Theorem Revisited}

\maketitle

\begin{abstract}
The classical Livšic theorem states that a real-valued observable over a transitive hyperbolic system is a coboundary if and only if its sums over all periodic orbits vanish. In this work, we extend this result to the non-invertible setting by establishing a $C^1$ Livšic cohomology theorem for Anosov endomorphisms. We also address the problem of characterizing when SRB measures coincide with inverse SRB measures in this setting. Finally, motivated by rigidity results for Anosov diffeomorphisms, we establish an analogous Livšic--Sinai-type theorem for $C^2$ Anosov endomorphisms. Specifically, we prove that if the restriction of an Anosov endomorphism $f$ to its non-wandering set has the same degree as $f$, and $Jf^n(p)=(\deg(f))^n$ for every periodic point $p$ of period $n$, then $f$ is transitive and preserves a normalized $C^1$ volume measure.
\end{abstract}

%----------------------------------------------------------
\section{Introduction}
%----------------------------------------------------------

The Livšic cohomological theorem is one of the celebrated results in the theory of hyperbolic dynamical systems. It gives a characterization, in terms of the values of a real-valued observable along periodic orbits, of when the observable can be represented as a coboundary. The theorem has applications in several areas, including thermodynamic formalism \cite{Bow75}, the construction of regular topological conjugacies \cite{dLL}, and the study of invariant measures \cite{LS72}, among others.

Anatole Livšic first formulated this cohomological framework in the early 1970s for Anosov systems \cite{L72}, establishing a criterion for solving the cohomological equation over transitive hyperbolic diffeomorphisms.

The first main result of this work establishes a $C^1$ Livšic cohomology theorem for non-invertible hyperbolic dynamics. Here $M$ is a compact and connected $C^{\infty}$ Riemannian manifold without boundary.

\begin{maintheorem}[Regularity of Solutions of Cohomological Equations]\label{TeoB}
	Let \mbox{$f:M\to M$} be a transitive Anosov endomorphism and \mbox{$\psi:M\to\mathbb{R}$} a Hölder continuous function. Suppose that for every $x\in M$ such that $f^n(x)=x$, we have \mbox{$\sum_{i=0}^{n-1}\psi(f^i(x))=0.$}
	\begin{itemize}
		\item[(a)] There exists a Hölder continuous function $\varphi:M\to \mathbb{R}$ such that $\psi=\varphi\circ f-\varphi$, and $\varphi$ is unique up to an additive constant.
		\item[(b)] If $\psi$ is $C^1$, then $\varphi$ is also $C^1$.
	\end{itemize}
\end{maintheorem}

The main contribution of the above theorem is item (b), which establishes the $C^1$ regularity of the solution.

We apply the theorem above to investigate a classical problem concerning SRB measures. A measure is called an SRB measure if its conditional measures along unstable manifolds are absolutely continuous with respect to the corresponding leafwise Lebesgue measures. Similarly, a measure is called an inverse SRB measure if its conditional measures along stable manifolds are absolutely continuous with respect to the corresponding leafwise Lebesgue measures.

A classical question asks whether a measure that is simultaneously an SRB measure and an inverse SRB measure must be equivalent to the volume measure. Bowen \cite{Bow75} answered this question affirmatively for Anosov diffeomorphisms, while Ledrappier \cite{ledrappier1984} obtained related results for non-uniformly hyperbolic systems. For non-invertible systems, a partial argument was given by Pei-Dong Liu (see Remark 2.8 in \cite{PDL2}), but a complete proof is still lacking. In the setting of Anosov endomorphisms, we establish the following result:

\begin{maintheorem}\label{Teo-Novo}
	Let $f: M \to M$ be a topologically transitive $C^2$ Anosov endomorphism. If the forward SRB measure ($\mu^+$) and the inverse SRB measure ($\mu^-$) coincide on the base manifold $M$, that is, $\mu^+ = \mu^- = \mu$, then:
	\begin{enumerate}
		\item The Jacobian $J(f^n(p)) = d^n$ for every periodic point satisfying $f^n(p) = p,$ where $d = \deg(f)$ is the degree of $f.$
		\item The measure $\mu$ is equivalent to the Riemannian volume $m$ on $M$.
		\item The density $\rho = \frac{d\mu}{dm}$ belongs to $C^1(M)$, thereby defining a smooth volume form.
	\end{enumerate}
\end{maintheorem}

While the converse of this theorem holds for diffeomorphisms (that is, if $f$ is a $C^2$ volume-preserving Anosov diffeomorphism, then the volume is SRB and inverse SRB), we show in Theorem \ref{Teo-exem} below that it fails in general for non-invertible maps.

After the publication of \cite{L72}, Anatole Livšic, in collaboration with Yakov Sinai, leveraged the techniques established therein to derive a fundamental rigidity result regarding Jacobians \cite{LS72}. Their classic theorem asserts that if the Jacobian of a transitive $C^2$ Anosov diffeomorphism equals one at every periodic point, the system must preserve a smooth volume form.

The classic Livšic-Sinai Theorem (reviewed, for instance, in \cite[Theorem 4.14]{Bow75} or \cite[Theorem 19.2.7]{KH95}) can be stated as follows:

\begin{theorem}[Livšic-Sinai, \cite{LS72}]\label{Teo1}
	Let $f:M \to M$ be a transitive $C^2$ Anosov diffeomorphism. Then, $f$ preserves an invariant measure $\nu$ absolutely continuous with respect to the Lebesgue measure $m$ if and only if $Jf^n(p)=1$ for every periodic point $p\in M$ of period $n.$
\end{theorem}

As a corollary of Theorem \ref{Teo-Novo}, we obtain a generalization of Theorem \ref{Teo1} for Anosov endomorphisms. For endomorphisms, the corresponding periodic condition is:
$Jf^n(p)=(\deg(f))^n,$
accounting for how the topological degree of the map scales the volume expansion.

\begin{corolA1}
	Let $f: M \to M$ be a topologically transitive $C^2$ Anosov endomorphism of topological degree $d \ge 1$. If, for every periodic point $p \in M$ of period $n$ (that is, $f^n(p) = p$), the total Jacobian satisfies the condition
	\begin{equation}
		 Jf^n(p)  = d^n,
	\end{equation}
	then $f$ preserves a measure $\nu$ that is mutually absolutely continuous (equivalent) with respect to the Riemannian volume $m$ on $M$. Moreover, the density $\frac{d\nu}{dm}$ belongs to $C^1(M)$.
\end{corolA1}

A compelling question that arises is whether transitivity must be assumed a priori in this context. A recent finding in \cite{M22} demonstrated that the condition on periodic Jacobians inherently forces the system to be transitive. Thus, for diffeomorphisms, transitivity follows from the periodic Jacobian condition. This leads to the following formulation of the result:

\begin{theorem}[\cite{M22}]\label{Teo2}
	Let $f:M \to M$ be a $C^2$ Anosov diffeomorphism. If the Jacobian satisfies $Jf^n(p)=1$ for all periodic points $p\in M$ with $f^n(p)=p$, then $f$ is topologically transitive and preserves a $C^1$ volume form.
\end{theorem}

Another goal of this work is to show that an analogous phenomenon occurs for non-invertible Anosov endomorphisms.  We establish that this single periodic constraint guarantees both topological transitivity and the existence of an invariant $C^1$ volume form.

Lastly, we recall that a subset $\Lambda\subset M$ is defined as \textit{completely $f$-invariant} if it satisfies the relation $f(\Lambda)=\Lambda=f^{-1}(\Lambda).$

%\section{Main Results}

 In this setting, the degree of a local homeomorphism $f,$ denoted by $\deg(f),$ is well-defined as the common number of preimages of each point of $M.$ We establish the following results.

\begin{maintheorem}[%Generalized
	Liv\v{s}ic--Sinai Theorem for Endomorphisms]\label{TeoA}
Let $f: M \to M$ be a $C^2$ Anosov endomorphism with a completely invariant non-wandering set $\Omega(f).$ If $Jf^n(p) = (\deg(f))^n = d^n$ for all $p \in M$ satisfying $f^n(p) = p,$ then $f$ is transitive. Furthermore, $f$ preserves a (unique) $C^1$ normalized volume form.
\end{maintheorem}

It is interesting to note that Theorem \ref{Teo2} is a corollary of Theorem \ref{TeoA}. Moreover, we can question whether the converse of this theorem is also true, as is the case with Theorems \ref{Teo1} and \ref{Teo2}. That is, if $f$ is volume preserving, then  $Jf^n(p)=d^n$, for all $p\in Per(f)$? In the non-invertible Anosov setting, the answer is negative.

\begin{maintheorem}\label{Teo-exem}
There is a $C^{\infty}$ volume-preserving Anosov endomorphism $f:\T^3\to\T^3$  with $Jf(p)\neq \deg(f)$ for a fixed point $p \in \T^3.$ Moreover, such a volume is SRB, but it is not inverse SRB.
\end{maintheorem}

This result highlights a fundamental difference between invertible and non-invertible dynamical systems: in the non-invertible setting, volume preservation alone does not guarantee that a measure is inverse SRB. By contrast, for a $C^2$ diffeomorphism, every invariant volume measure is an SRB measure for both $f$ and $f^{-1}$.

In another direction, we investigate the Lyapunov exponents on the set of periodic points of $f$.
Let $R_0(f)$ be the set of Lyapunov regular points of $f$. Given $x \in R_0(f )$, we denote by $\Lambda^u_f(x)$ the sum of all positive Lyapunov exponents at $x$. More precisely,
$$
\Lambda^u_f(x)=\displaystyle\sum_{i=1}^{l}\max\{\lambda_i(x),0\}\cdot \dime(E_i(x)).
$$

Analogously, we denote by $\Lambda^s_f(x)$ the sum of all negative Lyapunov exponents at $x \in R_0(f)$.

As an application of Theorem \ref{TeoA}, we show that if these sums are constant on periodic points, then the system is transitive.

\begin{maintheorem}\label{TeoC}
	Let $f:M\to M$ be a $C^2$-Anosov endomorphism with $\Omega(f)$ completely invariant and $\deg(f)=d$. If  $\Lambda^u_f$ and  $\Lambda^s_f$ are constant on $\Per(f)$, then	 $\Lambda^u_f + \Lambda^s_f=\log d$ and $f$ is transitive.
\end{maintheorem}

Regarding the hypotheses, in the results above $\Omega(f)$ is completely invariant. It is not difficult to verify that for Anosov endomorphisms one always has $f(\Omega)=\Omega$. However, it is not known in general whether $f^{-1}(\Omega)=\Omega$. We observe that
linear Anosov endomorphisms of the torus $\T^d$ provide canonical examples satisfying this property. This type of hypothesis appears frequently in the literature; see, for instance, \cite{LPPR24, M-2010, UW04}. It is well known that if $f$ is either transitive or conservative, then $\Omega(f)=M$. In particular,  $\Omega(f)$ is completely invariant. Therefore, our goal can be viewed as investigating the converse implication.

%\section{Organization of the Article}
%
%To facilitate the exposition, we divide the preliminaries into two parts. The first part addresses the foundational definitions and general properties of hyperbolic dynamics for endomorphisms. In the second part, we develop the essential theory regarding their SRB and inverse SRB measures.
%
%The remainder of the paper is dedicated to the proofs of our main results. While Theorem C is established directly within the Introduction, the remaining proofs are structured as follows: we first prove Theorem B by relying on Theorem A; we then proceed with the proofs of Theorems D and A. %The latter establishes the classical Liv\v{s}ic cohomological theorem in the setting of Anosov endomorphisms.

\section{Notions on Hyperbolicity in the Endomorphism Setting}
This section reviews the basic framework of hyperbolicity for non-invertible maps, introduced in the seminal works of Przytycki \cite{PRZ} and Ma\~n\'e-Pugh \cite{MP75}. These notions extend the corresponding concepts from the theory of diffeomorphisms to the non-invertible setting. To study such systems, we begin by introducing the topological construction of the inverse limit space, which allows us to recover a form of invertibility.

% ==========================================
% 1. LIMITE INVERSO (INVERSE LIMIT SPACE)
% ==========================================
\subsection{The Inverse Limit Space}

To resolve the multi-valued structure arising from non-invertibility, it is standard to lift the dynamics to the \textit{inverse limit space} (often called the solenoid), which collects all possible bi-infinite histories. Let $f:M\to M$ be  a continuous surjection.
\begin{equation*}
M^f = \left\{ \bar{x} = (x_n)_{n\in \mathbb{Z}} \in \prod_{i\in\mathbb{Z}} M : f(x_n) = x_{n+1} \text{ for all } n \in \mathbb{Z} \right\}.
\end{equation*}

Assuming $M$ is a compact Riemannian manifold, the space $M^f$ becomes a compact metrizable space when equipped with the metric:
\begin{equation*}
\bar{d}(\bar{x},\bar{y}) = \sum_{i\in\mathbb{Z}} 2^{-|i|} d(x_i,y_i),
\end{equation*}
where $d$ is the standard Riemannian distance on $M$.

The non-invertible map $f$ is replaced on $M^f$ by the \textit{shift homeomorphism} \mbox{$\bar{f}: M^f \to M^f$}, defined by $\bar{f}(\bar{x})_n = x_{n+1}$. The natural coordinate projection onto the present state, $\pi: M^f \to M$ given by $\pi(\bar{x}) = x_0$, is continuous and provides the semiconjugacy relation $f \circ \pi = \pi \circ \bar{f}$.

The inverse limit space $M^f$ is a type of fiber bundle, in which its fibers are Cantor sets.

\begin{theorem}[\cite{AH94}, Theorem 6.5.1]\label{Teo-Aoki-Invese-Limit}
Let $f:M\to M$ be a continuous surjection of degree greater than one. Then $(M^f,M,C,p)$ is a fiber bundle where $C$ denotes the Cantor set.
\end{theorem}

Moreover, topological transitivity behaves well under this lift.

\begin{theorem}\label{Teo-transitive}
Let $f:M\to M$ be a continuous surjection. Then $f$ is topologically transitive if and only if the shift map $\bar{f}:M^f\rightarrow M^f$ so is.
\end{theorem}

Given any subset $\Lambda\subset M$, we can naturally define its lift to the solenoid as:
\begin{equation*}
\Lambda^f=\{\bar{x} \in M^f : x_n\in\Lambda \text{ for all } n\in\Z\}.
\end{equation*}

% ==========================================
% 2. CONJUNTOS HIPERBÓLICOS (HYPERBOLIC SETS)
% ==========================================
\subsection{Hyperbolic Sets}

With the inverse limit space established, we can define uniform hyperbolicity for invariant subsets of non-invertible maps. Unlike the diffeomorphism setting, the hyperbolic splitting must be defined along complete history paths.

\begin{definition}
Let $f:M\rightarrow M$ be a $C^1$ local diffeomorphism. A compact, non-empty invariant subset $\Lambda\subset M$ (i.e., $f(\Lambda)=\Lambda$) is called a \emph{hyperbolic set} for $f$ if for every history $\bar{x} \in \Lambda^f$ there exists a continuous $Df$-invariant splitting
$T_{x_i}M=E_{x_i}^s\oplus E_{x_i}^u$ for all $i\in\Z$, and constants $C>0$, $\lambda>1$ satisfying, for all $n \geq 0$:
\begin{enumerate}		
	\item $||Df^n_{x_i}(v^u)||\geq C^{-1}\lambda^n||v^u||, \quad \text{for all } v^u\in E^u_{x_i},$
	\item $||Df^n_{x_i}(v^s)||\leq C\lambda^{-n}||v^s||, \quad \text{for all } v^s\in E^s_{x_i}.$
\end{enumerate}	
\end{definition}

% ==========================================
% 3. AXIOMA A (AXIOM A ENDOMORPHISMS)
% ==========================================
\subsection{Axiom A Endomorphisms}

Using the general notion of hyperbolic sets, we can formulate the non-invertible version of Axiom A systems, which describes global dynamics structured around hyperbolic behavior on the non-wandering set.

\begin{definition}[Definition 3.3, \cite{PRZ}]
A $C^1$ local diffeomorphism $f:M\rightarrow M$ is said to be an \emph{Axiom A} endomorphism if:
\begin{itemize}
	\item The non-wandering set $\Omega(f)$ is hyperbolic,
	\item The set of periodic points $Per(f)$ is dense in $\Omega(f)$.
\end{itemize}
\end{definition}

Axiom A endomorphisms inherit the spectral decomposition properties of the non-wandering set (see, e.g., Theorem 3.4.4 of \cite{AH94}):

\begin{theorem}[Spectral Decomposition]\label{Teo-Dec-Spec}
If $f$ is an Axiom A endomorphism, then its non-wandering set $\Omega(f)$ can be uniquely decomposed as a finite disjoint union of compact, invariant, and topologically transitive sets:
$$\Omega(f)=\bigcup_{i=1}^{l}\Omega_i,$$
where $f(\Omega_i)=\Omega_i$. These components $\Omega_i$ are the basic sets of $f$. Additionally, each basic set $\Omega_i$ admits a partition into finite disjoint compact subsets $\Omega_i=\bigcup_{j=1}^{n_i}\Omega_{i,j}$ such that $f(\Omega_{i,j})=\Omega_{i,j+1}$ (with $\Omega_{i,n_i+1}=\Omega_{i,1}$), where the restriction $f^{n_i}|_{\Omega_{i,j}}$ is topologically mixing.
\end{theorem}

\begin{definition}
Let $f$ be an Axiom A endomorphism. A basic set $\Lambda \subset \Omega(f)$ is classified as a \emph{repeller} if its stable set satisfies $W^s(\Lambda)=\Lambda$, which is equivalent to saying that $\Lambda$ possesses a neighborhood $U$ such that $\overline{U} \subset f (U)$.
\end{definition}

\begin{remark}\label{remk-reppeler}
Every Axiom A endomorphism has a repeller. Indeed, since $f$ is an Axiom A endomorphism, by Theorem 3.2.5 of \cite{AH94}, the induced map $\bar{f}$ on the inverse limit space $M^f$ has a repeller $\overline{B}$. Therefore, the set $B=\pi(\overline{B})$ is a repeller for $f$.	
\end{remark}

% ==========================================
% 4. ENDOMORFISMOS DE ANOSOV (ANOSOV ENDOMORPHISMS)
% ==========================================
\subsection{Anosov Endomorphisms}

Finally, we address the case where the entire manifold $M$ behaves as a hyperbolic set. This leads to the class of Anosov endomorphisms.

\begin{definition}[\cite{PRZ}]\label{def1}
Let $M$ be a smooth manifold and $f:M\rightarrow M$ a $C^1$ local diffeomorphism. We say that $f$ is an \emph{Anosov endomorphism} if there exist real constants $C>0$ and $\lambda>1$ such that, for every bi-infinite orbit $\bar{x} = (x_n)_{n\in{\Z}}$ (meaning $f(x_n)=x_{n+1}$ for all $i\in\Z$), there exists a continuous $Df$-invariant splitting of the tangent bundle:
$$T_{x_i}M=E_{x_i}^s\oplus E_{x_i}^u, \quad \forall i\in\Z$$
which, for every $n \geq 0$, satisfies the following contraction and expansion estimates:
\begin{enumerate}		
	\item $||Df^n_{x_i}(v^u)||\geq C^{-1}\lambda^n||v^u||, \quad \text{for all } v^u\in E^u_{x_i}$,
	\item $||Df^n_{x_i}(v^s)||\leq C\lambda^{-n}||v^s||, \quad \text{for all } v^s\in E^s_{x_i}$.
\end{enumerate}
\end{definition}

Because Anosov endomorphisms exhibit the shadowing property, they are topologically stable and satisfy the Axiom A conditions, as established in \cite{PRZ}.

\begin{proposition}[Proposition 3.2, \cite{PRZ}]\label{PropAA}
If $f:M\rightarrow M$ is an Anosov endomorphism, then $\overline{Per(f)}=\Omega(f)$. Consequently, every Anosov endomorphism is also an Axiom A endomorphism.
\end{proposition}

As in the invertible scenario, the hyperbolic splitting over orbit histories guarantees the existence of stable and unstable manifolds. Specifically, each fixed history $\bar{x} = (x_n)_{n\in{\Z}}$ is associated with local invariant manifolds characterized as follows:

\begin{theorem} [Theorem 2.1, \cite{PRZ}]\label{Teo-Loc-manifolds}
Let $f$ be an Anosov endomorphism. For any chosen history path $\bar{x}=(x_n)_{n\in{\Z}}$, the following assertions hold:
\begin{itemize}
	\item[a)] The local stable set, defined by
	$$W^s_{\delta}(\bar{x})=\{y\in M : d(f^n(x_0),f^n(y)) \leq {\delta}, \, \forall n \geq 0\},$$
	is an embedded submanifold of $M$;
	\item[b)] The local unstable set, defined by
	$$W^u_{\delta}(\bar{x})=\{y\in M : \exists \, (y_n)_{n=-\infty}^0 \text{ with } f(y_n)=y_{n+1}, \, y_0 = y, \text{ and } d(x_n,y_n)\leq {\delta}, \, \forall n\leq 0 \},$$
	is an embedded submanifold of $M$;
	\item[c)] The tangent spaces at the base point satisfy $$T_{x_0}W^s_{\delta}(\bar{x})=E^s(x_0) \,\, \text{and}\,\, T_{x_0}W^u_{\delta}(\bar{x})=E^u(x_0)$$.
\end{itemize}
\end{theorem}

A crucial phenomenon in non-invertible dynamics, which contrasts sharply with diffeomorphism theory, is the non-uniqueness of the unstable distributions at a given point. If $f$ is an Anosov endomorphism, two distinct orbit histories $\bar{x} = (x_n)_{n \in \mathbb{Z}}$ and $\bar{y} = (y_n)_{n \in \mathbb{Z}}$ sharing the same base point $x_0 = y_0$ will necessarily yield the same stable subspace, i.e., $E^s_f(x_0) = E^s_f(y_0)$. In contrast, their unstable subspaces $E^u_f(x_0)$ and $E^u_f(y_0)$ can be completely different, as they depend on the negative history of the orbit.

According to Theorem~\ref{Teo-Loc-manifolds}, these distinct subspaces integrate to different local unstable manifolds $W^u(\bar{x})$ and $W^u(\bar{y}).$ Consequently, a single point $x \in M$ typically acts as the base point for multiple, potentially intersecting local unstable leaves. While invariant manifolds for Anosov diffeomorphisms depend smoothly on the base point, in the endomorphism context this regular dependence is recovered by topologizing the families over the inverse limit space.

\begin{theorem}[Theorem 2.5, \cite{PRZ}]\label{PRZ-1}
Let $f:M\rightarrow M$ be a $C^1$ Anosov endomorphism. If a sequence of histories $\bar{x}_n \in M^f$ converges to $\bar{x}\in M^f$ in the inverse limit topology, then
$$
W^u_{\delta}(\bar{x}_n)\to W^u_{\delta}(\bar{x}) \quad \text{and} \quad W^s_{\delta}(\bar{x}_n)\to W^s_{\delta}(\bar{x})$$
with respect to the $C^1$ topology on compact subsets.
\end{theorem}

In addition to the inverse limit space $M^f,$ endomorphisms can be studied via their universal covering $\widetilde{M}$. These structures provide complementary ways of analyzing the dynamics and allow us to use the hyperbolic theory of diffeomorphisms for the lifted system. Regarding the lifts in the universal cover, we have:

\begin{proposition}[\cite{MP75}]\label{Prop-Mane}
Let $\widetilde{M}$ be the universal cover of $M$. Then $f$ is an Anosov endomorphism of $M$, if and only if, its lift $\tilde{f}: \widetilde{M} \rightarrow \widetilde{M}$ is an Anosov diffeomorphism.
\end{proposition}

\section{Generalities on SRB and Inverse SRB Measures}

In this section, we summarize the essential ergodic properties of physical and inverse SRB measures in the non-invertible $C^{1+\alpha}$ framework ($\alpha > 0$). These notions will play an important role in establishing the topological transitivity of the systems under consideration. For a comprehensive review of these foundational results, we refer the reader to Section 2 of \cite{M25} and the references therein.

To establish our main results, we first state a crucial relation regarding the lift of invariant measures, followed by the geometric construction of subordinate partitions.

A well-known result guarantees that any invariant probability measure on the base manifold corresponds to a unique invariant measure on the solenoid.

\begin{theorem}[\cite{QXZ}]\label{teo:lift-measure}
Let $(M,d)$ be a compact metric space and $f: M \rightarrow M$ a continuous map. Every $f$-invariant Borel probability measure $\mu$ induces a unique $\bar{f}$-invariant Borel probability measure $\bar{\mu}$ on the inverse limit space $M^f$ characterized by the relation $\bar{\mu}(\pi^{-1}(B)) = \mu(B)$ for all Borel sets $B \subset M$.
\end{theorem}

We now define the class of partitions whose atoms lie locally inside the unstable manifolds.

\begin{definition}\label{def:subordinate-partition}
We say that a measurable partition $\eta$ of $M^f$ is subordinate to the $W^u$-manifolds of the system $(f, \mu)$ if, for $\bar{\mu}$-almost all histories $\bar{x} \in M^f$, the atom $\eta(\bar{x})$ containing $\bar{x}$ satisfies the following properties:
\begin{enumerate}
    \item The restriction of the projection map, $\pi|_{\eta(\bar{x})}$, is a bijection from $\eta(\bar{x})$ onto its image $\pi(\eta(\bar{x}))$;
    \item There exists a $C^1$-embedded submanifold $W(\bar{x}) \subset M$ of dimension $k(\bar{x})$ such that $W(\bar{x}) \subset W^u(\bar{x})$, $\pi(\eta(\bar{x})) \subset W(\bar{x})$, and the image $\pi(\eta(\bar{x}))$ contains a neighborhood of $x_0$ that is open with respect to the submanifold topology of $W(\bar{x})$.
\end{enumerate}
\end{definition}

\begin{definition}
 Let $f:M\to M$ be a $C^{1+\alpha}$ endomorphism, with $\alpha > 0,$ and let $\mu$ be an $f-$invariant probability measure. We say that $\mu$ satisfies the SRB property if, for every measurable partition $\eta$ of $M^f$ subordinate to the unstable manifolds, the projected conditional measures $\pi_{\ast}\bar{\mu}_{{\eta(\bar{x})}}$ are absolutely continuous with respect to the induced Riemannian volume $m_{{\eta(\bar{x})}}$ on the corresponding unstable manifolds, for $\bar{\mu}-$almost every $\bar{x}\in M^f$; that is,
$$\pi_{\ast}\bar{\mu}_{\eta(\bar{x})}
\ll m_{\eta(\bar{x})}
\qquad\text{for }\bar{\mu}\text{-a.e. }\bar{x}\in M^f.$$
Such a measure $\mu$ is simply called an SRB measure.
\end{definition}

An alternative characterization of SRB measures is obtained via the
thermodynamic formalism, through the notion of topological pressure.
To this end, we consider the geometric potential
$\bar{\phi}\colon M^f\to\mathbb{R}$, defined by
\[
\bar{\phi}^u(\bar{x})
=
-\log\left|
\det\left(
Df_{x_0}\big|_{E^u(\bar{x})}
\right)
\right|,
\]
where $E^u(\bar{x})$ denotes the unstable subspace associated with
the orbit represented by $\bar{x}\in M^f.$ Under appropriate
dynamical hypotheses, the natural lift $\bar{\mu}$ of an SRB
measure can be characterized as an equilibrium measure for this
potential, satisfying
\[
h_{\bar{\mu^{+}}}(\bar{f})
+
\int_{M^f}\bar{\phi}^u\,d\bar{\mu^{+}}
=
P(\bar{f},\bar{\phi}^u)
=
0.
\]
Thus, the SRB measure is characterized as an invariant measure
that realizes the supremum in the variational principle for the
topological pressure $P$ associated with the geometric potential,
establishing a connection between metric entropy, expansion along
unstable directions, and the statistical structure of the
dynamical system. For a background see \cite{QXZ}.

We can summarize all this information as follows.

\begin{theorem}
Let $f:M\rightarrow M$ be a $C^{1+\alpha}$ endomorphism on a compact Riemannian manifold $M$. An $f$-invariant probability measure $\mu^{+}$ is an \emph{SRB measure} if and only if it satisfies the classical Pesin entropy formula:
\begin{equation}\label{eq:pesin_srb}
h_{\mu^{+}}(f)=\int_{M} \Lambda^u_f(x) d\mu^{+},
\end{equation}
where $\Lambda^u_f(x) = \sum_{i} \max\{\lambda_i(x), 0\} \dim(E_i(x))$ represents the sum of positive Lyapunov exponents counted with multiplicity.\end{theorem}

Geometrically, such measures are characterized by having conditional measures on partitions subordinate to unstable manifolds that are absolutely continuous with respect to the Lebesgue measure on those leaves, possessing positive continuous densities (see \cite{QZ}).

For stable manifolds, there is an analogous formulation of measurable partition subordinate to stable leaves, see \cite{PDL2}. On the other hand, for a completely invariant hyperbolic repeller $\Lambda \subset M$ of topological degree $d$, one can formulate the notion of an \emph{inverse SRB measure} $\mu^{-}$. Under the assumption of $C^{1+\alpha}$ regularity, there is a unique invariant probability measure $\mu^{-}$ is identified as the inverse SRB measure if it satisfies the inverse Pesin entropy relation:
\begin{equation}\label{eq:inverse_pesin}
h_{\mu^{-}}(f) = \log d - \int_{M} \Lambda^s_f(x) d\mu^{-},
\end{equation}
where $\Lambda^s_f(x) = \sum_{i} \min\{\lambda_i(x), 0\} \dime(E_i(x))$ denotes the sum of negative Lyapunov exponents. In this setting, the measure $\mu^{-}$ is characterized by having conditional measures on partitions subordinate to the stable manifolds that are equivalent to the corresponding leaf volume, with positive $C^0$ densities (see \cite{M-2010,PDL2}).

An alternative characterization of inverse SRB measures is
provided by the thermodynamic formalism. Let
$f\colon M\to M$ be a smooth endomorphism and let
$\Lambda$ be a hyperbolic repeller. The stable potential is
defined by
\[
\phi^s(x)
=
\log\left|
\det\left(Df_x|_{E^s(x)}\right)
\right|.
\]
Under the appropriate hypotheses, the inverse SRB measure
$\mu^-$ is an equilibrium measure for $\phi^s$. If $f$ is
$d$-to-$1$ on $\Lambda$, then the topological pressure satisfies
\[
P(\phi^s)=\log d.
\]
Equivalently, the normalized potential has zero topological
pressure:
\[
P(\phi^s-\log d)=0.
\]
For a background see \cite{M-2010}.

\begin{remark}\label{rem:srb_degree}
The complete invariance of the repeller $\Lambda$ guarantees that the degree
$d(\cdot)=d$ is constant throughout the dynamics. Consequently, no
connectedness assumption on the repeller itself is needed to establish the
existence and uniqueness of $\mu^{-}$ (see \cite{M-2010}, p.~803).
\end{remark}

\section{Cohomological equations and proof of Theorem \ref{TeoB}}

Cohomological equations play an important role in dynamical systems. They arise naturally in several contexts, and their applications include the construction of conjugacies, the study and computation of densities of invariant measures, and  rigidity problems.

Given  maps $f:M\to M$ and $\psi:M\to \R$, we say that a map $\varphi:M\to\R$ is a coboundary for $(f,\psi)$ if it is a solution of the cohomological equation, in the functional variable $u$
\begin{equation}\label{eqcohomological}
	\psi(x)=u(f(x))-u(x).
\end{equation}
If $u=\varphi$ is a solution, then by induction we obtain
\begin{equation}\label{eqSum1}
	\varphi(f^n(x))=\varphi(x)+\sum_{j=0}^{n-1}\psi(f^{j}(x)), \,\,\, n\in\N.
\end{equation}

If $f$ is invertible, we also have

\begin{equation}\label{eqSum2}
	\varphi(f^{-n}(x))=\varphi(x)-\sum_{j=1}^{n}\psi(f^{-j}(x)), \,\,\, n\in\N.
\end{equation}

We remark that if $p\in M$ satisfies $f^n(p)=p$, thus
\begin{equation}
	0=\varphi(f^n(p))-\varphi(p)=\sum_{j=0}^{n-1}\psi(f^{j}(p)).
\end{equation}

Therefore, a necessary condition for the existence of a solution to the cohomological equation \eqref{eqcohomological} is that
 $\sum_{j=0}^{n-1}\psi(f^{j}(p))=0$ for every $n$-periodic point $p\in M$.

%-----------------------------

%\begin{remark}\label{Remk-Liv}
%
%\end{remark}

\begin{proof}[Proof of Theorem \ref{TeoB}]

We emphasize that Liv\v{s}ic's Theorem for Hölder continuous observables $\psi$ also holds in the context of endomorphisms. Once exponential estimates for the Anosov Closing Lemma for endomorphisms have been established (see Appendix, Section \ref{appendix}), the proof is completely analogous to the invertible case, using a point with a dense positive orbit. For the proof in the setting of Anosov diffeomorphisms, we refer the reader to Theorem 19.2.1 in \cite{KH95}.

For the $C^1$ case, constructing the argument directly on the manifold $M$ is delicate because unstable sets may intersect or fold, allowing infinitely many unstable branches through a single point. To address this tangling, we adopt a local approach by lifting the dynamics to the universal covering $\widetilde{M}.$
	
Proposition \ref{Prop-Mane} plays a crucial role in our proof. By passing to the universal cover, we obtain a favorable global structure: any lift $\tilde{f}$ of $f$ acts as a diffeomorphism with well-defined, globally transverse stable and unstable foliations. From now on, we fix a lift $\tilde{f}:\widetilde{M}\to\widetilde{M}$ of $f.$
	
As observed before, the existence of a Lipschitz solution $\varphi$ is obtained similarly to the invertible case, see the appendix, section \ref{appendix}. It remains to show that $\varphi$ is a $C^1$ function.
	
	In the universal covering $\widetilde{M}$, fix a fundamental domain $R\subset\widetilde{M}$ and consider the canonical projection $p:\widetilde{M}\to M$ alongside its restriction $\hat{p}=p|_R$. Define the lifted functions $\tilde{\psi}$ and $\tilde{\varphi}$ on $\widetilde{M}$ by $\tilde{\psi}(\tilde{x})=\psi(p(\tilde{x}))$ and $\tilde{\varphi}(\tilde{x})=\varphi(p(\tilde{x}))$. It follows immediately that $\tilde{\psi}(\hat{p}^{-1}(x))=\psi(x)$ and $\tilde{\varphi}(\hat{p}^{-1}(x))=\varphi(x)$ for all $x \in p(R).$ Moreover, $\tilde{f}, \tilde{\psi},$ and $\tilde{\varphi}$ satisfy the lifted cohomological equation:
	$$ \tilde{\psi}(\tilde{x}) = \tilde{\varphi}(\tilde{f}(\tilde{x})) - \tilde{\varphi}(\tilde{x}).$$
	
	Let $\widetilde{W}^u$ and $\widetilde{W}^s$ denote the unstable and stable foliations of the diffeomorphism $\tilde{f},$ respectively. We will first show that $\tilde{\varphi}$ is $C^1$ along the leaves of the unstable foliation $\widetilde{W}^u$.
	
	Take $\tilde{x}$ and $\tilde{y}$ in the same unstable manifold. Using the cohomological equation, we can express $\tilde{\varphi}$ in terms of its past iterates. By substituting $\tilde{x}$ and $\tilde{y}$ into the equation and iterating $n$ times backwards, it is straightforward to verify that:
	$$
	\tilde{\varphi}(\tilde{x})-\tilde{\varphi}(\tilde{y}) = \left[ \sum_{j=1}^{n}\tilde{\psi}(\tilde{f}^{-j}(\tilde{x})) - \sum_{j=1}^{n}\tilde{\psi}(\tilde{f}^{-j}(\tilde{y})) \right] + \tilde{\varphi}(\tilde{f}^{-n}(\tilde{x})) - \tilde{\varphi}(\tilde{f}^{-n}(\tilde{y})),
	$$
	for all $n\in\mathbb{N}$. Since $\tilde{y} \in \widetilde{W}^u(\tilde{x})$, the distance $d(\tilde{f}^{-n}(\tilde{x}),\tilde{f}^{-n}(\tilde{y}))\to 0$ as $n \to \infty$. As $\tilde{\varphi}$ is uniformly continuous (being the lift of a Lipschitz function on a compact manifold), the last two terms cancel in the limit. Taking $n\to\infty$, we obtain:
	\begin{equation}\label{eqsumc}
		\tilde{\varphi}(\tilde{x})-\tilde{\varphi}(\tilde{y}) = \sum_{j=1}^{+\infty} \Big( \tilde{\psi}(\tilde{f}^{-j}(\tilde{x})) - \tilde{\psi}(\tilde{f}^{-j}(\tilde{y})) \Big).
	\end{equation}
	
	To show that $\tilde{\varphi}$ is continuously differentiable along the unstable leaves, we differentiate Equation (\ref{eqsumc}) formally. Fix a unit vector $v\in \widetilde{E}^u_{\tilde{x}}$ and take a smooth curve $\gamma:(-\varepsilon,\varepsilon)\to \widetilde{W}^u(\tilde{x})$ such that $\gamma(0)=\tilde{x}$ and $\gamma'(0)=v$. Replacing $\tilde{x}$ by $\gamma(t)$ in the series, we consider the formal derivative of the terms with respect to $t$ at $t=0$:
	$$
	\sum_{j=1}^{+\infty} \frac{d}{dt} \Big( \tilde{\psi}\circ \tilde{f}^{-j}\circ\gamma(t) \Big)\Big|_{t=0} = \sum_{j=1}^{+\infty} D\tilde{\psi}(\tilde{f}^{-j}(\tilde{x})) D\tilde{f}^{-j}(\tilde{x}) v.
	$$
	
	We first show that $\tilde{\varphi}$ is $C^1$ along the unstable leaves. Since $\tilde{f}$ is an Anosov diffeomorphism, there exist constants $C>0$ and $\lambda\in(0,1)$ such that
\[
\|D\tilde{f}^{-j}(\tilde{x})v\|
\le C\lambda^j\|v\|,
\qquad j\ge 1,
\]
for every $\tilde{x}\in\widetilde{M}$ and $v\in\widetilde{E}^u_{\tilde{x}}$. Moreover, since $\psi$ is $C^1$ on the compact manifold $M$, its lift $\tilde{\psi}$ has uniformly bounded derivative. Let $K>0$ be such that $\|D\tilde{\psi}\|\le K$.

Consider the partial sums
\[
\tilde{\varphi}_n(\tilde{x})
=\sum_{j=1}^{n}\tilde{\psi}(\tilde{f}^{-j}(\tilde{x})).
\]
The hyperbolicity estimates imply that the defining series converges uniformly. Furthermore, for every unit vector $v\in\widetilde{E}^u_{\tilde{x}}$ and $n>m$, we have
\[
\begin{aligned}
&\left\|D\tilde{\varphi}_n(\tilde{x})v
-D\tilde{\varphi}_m(\tilde{x})v\right\|\\
&\qquad\le KC\sum_{j=m+1}^{n}\lambda^j.
\end{aligned}
\]
Since the geometric series converges, the derivatives of the partial sums converge uniformly along unstable leaves. By the standard theorem on the convergence of differentiable functions, $\tilde{\varphi}$ is therefore $C^1$ along unstable leaves, with
\[
D\tilde{\varphi}(\tilde{x})v
=\sum_{j=1}^{\infty}
D\tilde{\psi}(\tilde{f}^{-j}(\tilde{x}))
D\tilde{f}^{-j}(\tilde{x})v.
\]

An analogous argument, using forward iterates and the exponential contraction of vectors in the stable bundle, shows that $\tilde{\varphi}$ is $C^1$ along stable leaves.

To obtain global regularity, we apply the following classical lemma, often attributed to Journé.

\begin{lemma}[\cite{KH95}, Lemma 19.1.10]\label{Lema7}
Suppose $\varphi:\mathbb{R}^n\to\mathbb{R}$ is $C^1$ along the leaves of two continuous transverse foliations $W^u$ and $W^s$ in $\mathbb{R}^n$. Then $\varphi$ is locally $C^1$ globally.
\end{lemma}

Applying this lemma, we conclude that $\tilde{\varphi}$ is $C^1$ on $\widetilde{M}$. Since the covering projections $p$ and $\hat{p}$ are smooth local diffeomorphisms, it follows that $\varphi=\tilde{\varphi}\circ\hat{p}^{-1}$ is also $C^1$. This completes the proof of Theorem \ref{TeoB}.\end{proof}

\section{Proof of Theorem \ref{Teo-Novo}}
\begin{proof}
\textbf{Step 1: Gibbs Formalism on the Inverse Limit.} Define the Hölder-continuous geometric potentials associated with the hyperbolic directions on the smooth manifold $M$:
\begin{align}
    \bar{\phi}^u(\bar{x}) &= -\log\left|\det\left(Df_{x_0}\big|_{E^u(\bar{x})}\right)\right|, \\
    \bar{\phi}^s(\bar{x}) &= \log\left|\det\left(Df_{x}\big|_{E^s(x)}\right)\right|,
\end{align}
where $E^u(\bar{x})$ denotes the unstable subspace associated with
the orbit represented by $\bar{x}\in M^f.$ We introduce the notation $J^uf(\bar{x}) = \left|\det\left(Df_{x_0}\big|_{E^u(\bar{x})}\right)\right|$ and \mbox{$J^sf(\bar{x}) = J^sf(x) =  \left|\det\left(Df_{x}\big|_{E^s(x)}\right)\right|.$}

 By Bowen's classical theory \cite{Bo74}, since $\bar{f}$ has the specification property as a consequence of transitivity (Theorem 11.5.13 \cite{AH94}), there are unique equilibrium states $\bar{\mu}^+$ and $\bar{\mu}^-$ on $M^f$ for these potentials, characterized as \textbf{Gibbs states} with zero topological pressure ($P(\bar{\phi}^u) = P(\bar{\phi}^s) = 0$).

For sufficiently small $\varepsilon > 0$, there exists a constant $C_0 \ge 1$ such that, for every dynamical ball $B_n(\bar{x}, \varepsilon) \subset M^f$, the measures satisfy the following uniform estimates for all $n \ge 1$:
\begin{equation}
    C_0^{-1} \le \frac{\bar{\mu}^+(B_n(\bar{x}, \varepsilon))}{\exp\left(S_n \bar{\phi}^u(\bar{x})\right)} \le C_0 \quad \text{and} \quad C_0^{-1} \le \frac{\bar{\mu}^-(B_n(\bar{x}, \varepsilon))}{\exp\left(S_n \bar{\phi}^s(\bar{x})\right)} \le C_0.
\end{equation}
For such estimates we refer to \cite{Bow75, Bo74} and \cite{M-2010}.

\textbf{Step 2: Comparing Birkhoff Sums.} By definition, the measures on the base are the projections $\pi_* \bar{\mu}^+ = \mu^+$ and $\pi_* \bar{\mu}^- = \mu^-$. Since $\mu^+ = \mu^- = \mu$ by hypothesis, and since the map $\pi_*$ is injective on the space of invariant measures, it follows that the lifts are identical on the inverse limit: $\bar{\mu}^+ = \bar{\mu}^- = \bar{\mu}$. Substituting this identity into the Gibbs inequalities from Step 1, we obtain, for every orbit $\bar{x} \in M^f$ and every $n \ge 1$:
\begin{equation}
|S_n \bar{\phi}^u(\bar{x}) - S_n \bar{\phi}^s(\bar{x})| \le 2\log C_0.
\end{equation}

\textbf{Step 3: Projection to Periodic Points.} Let $\bar{p} \in M^f$ be a periodic point of period $n$ for the shift ($\bar{f}^n(\bar{p}) = \bar{p}$). Evaluating the preceding inequality along the orbit of $\bar{p}$ for multiples $kn$ of the period, and using the fact that Birkhoff sums grow linearly along periodic points ($S_{kn}\bar{\phi}(\bar{p}) = k \cdot S_n\bar{\phi}(\bar{p})$), we obtain:
\begin{equation}
 |k \cdot S_n\bar{\phi}^u(\bar{p}) - k \cdot S_n\bar{\phi}^s(\bar{p})| \le 2\log C_0 \implies |S_n\bar{\phi}^u(\bar{p}) - S_n\bar{\phi}^s(\bar{p})| \le \frac{2\log C_0}{k}.
\end{equation}
Taking the limit as $k \to \infty$, we deduce that the difference is identically zero: $S_n\bar{\phi}^u(\bar{p}) = S_n\bar{\phi}^s(\bar{p})$. Projecting this equality to the smooth base manifold $M$ via the periodic point $p = \pi(\bar{p})$ (which satisfies $f^n(p)=p$), and expanding the original potentials, we obtain:
\begin{equation}
    -\sum_{i=0}^{n-1} \log J^u f(f^i(p)) = \sum_{i=0}^{n-1} \left( \log J^s f(f^i(p)) - \log d \right).
\end{equation}
Grouping the terms and invoking the classical hyperbolic transversality on the base for the total Jacobian ($Jf = J^u f \cdot J^s f$), we obtain:
\begin{equation}
    \sum_{i=0}^{n-1} \log Jf(f^i(p)) = n \log d \implies Jf^n(p) = d^n, \quad \forall p \in \operatorname{Per}(f).
\end{equation}

\textbf{Step 4: Smooth Regularity.} Define the potential $\psi(x) = \log Jf(x) - \log d$ on the smooth manifold $M$. By Step 3, the Birkhoff sums of $\psi$ vanish on every periodic orbit in $M$, that is, $S_n \psi(p) = 0$.

Applying the generalized Liv\v{s}ic Theorem \ref{TeoB}, since $f$ is $C^2$ and satisfies the periodic obstruction condition for the Jacobian, the cohomological equation
\begin{equation}
    \log Jf(x) - \log d = u(f(x)) - u(x)
\end{equation}
admits a solution $u: M \to \mathbb{R}$ that is necessarily of class $C^1(M).$

\textbf{Step 5: Invariance and Equivalence to Volume.} Define the measure $\nu$ on $M$ by the density $\rho(x) = e^{u(x)}$ with respect to Riemannian volume $m$, that is, $d\nu = e^{u(x)}dm$. Since $M$ is compact and $u \in C^1(M)$, there are constants \mbox{$0 < c_1 \le e^{u(x)} \le c_2 < \infty$}, which establishes mutual equivalence between $\nu$ and $m$. The invariance of $\nu$ follows directly from the change-of-variables formula for covering maps and the coboundary relation established above:
\begin{equation}
    \int_M \varphi(f(x)) d\nu(x) = \int_M \varphi(y) \left( \sum_{x \in f^{-1}(y)} \frac{e^{u(x)}}{Jf(x)} \right) dm(y).
\end{equation}
Substituting $\frac{1}{Jf(x)} = \frac{1}{d} e^{u(x) - u(f(x))}$, we obtain:
\begin{equation}
    \sum_{x \in f^{-1}(y)} \frac{e^{u(x)}}{d \cdot e^{u(x) - u(y)}} = \sum_{x \in f^{-1}(y)} \frac{e^{u(y)}}{d} = e^{u(y)} \left( \frac{1}{d} \sum_{x \in f^{-1}(y)} 1 \right) = e^{u(y)},
\end{equation}
since $\# f^{-1}(y) = d$ for almost every $y$. Hence, $\int_M \varphi(f(x)) d\nu(x) = \int_M \varphi(y)d\nu(y)$, proving that $\nu$ is an $f$-invariant measure equivalent to Riemannian volume, with a $C^1$ density. By the uniqueness of the measure absolutely continuous with respect to volume for transitive Anosov systems, we conclude that $\mu = \nu$.
\end{proof}

\begin{corollary}[Corollary A1]\label{Cor1}
Let $f: M \to M$ be a topologically transitive $C^2$ Anosov endomorphism of topological degree $d \ge 1$. If, for every periodic point $p \in M$ of period $n$ (that is, $f^n(p) = p$), the total Jacobian satisfies the condition
\begin{equation}
     Jf^n(p)  = d^n,
\end{equation}
then $f$ preserves a measure $\nu$ that is mutually absolutely continuous (equivalent) with respect to the Riemannian volume $m$ on $M$. Moreover, the density $\frac{d\nu}{dm}$ belongs to $C^1(M)$.
\end{corollary}

\begin{proof}
Consider the function $\psi(x) = \log  Jf(x)  - \log d$ defined on the smooth manifold $M$. Since $f$ is a $C^2$ endomorphism, the stable and unstable subbundles are Hölder-continuous, making the partial Jacobians and the function $\psi(x)$ Hölder-continuous on $M$.

By hypothesis, for any periodic point $p \in M$ with $f^n(p) = p$, we have $ Jf^n(p)  = d^n$. Taking logarithms on both sides and expanding the Jacobian of the composition by the chain rule gives:
\begin{equation}
    \sum_{i=0}^{n-1} \log  Jf(f^i(p))  = n\log d \implies \sum_{i=0}^{n-1} \left( \log Jf(f^i(p))  - \log d \right) = 0.
\end{equation}
This shows that the Birkhoff sum of $\psi$ vanishes identically on the set of all periodic points of $M$, that is, $S_n \psi(p) = 0$ for every $p \in \operatorname{Per}(f)$. Since the periodic obstruction vanishes, we directly invoke the generalized Liv\v{s}ic Theorem \ref{TeoB}. By the $C^2$ regularity of the endomorphism and the vanishing of the periodic sums, there exists a transfer function $u: M \to \mathbb{R}$ of class $C^1(M)$ satisfying the cohomological equation:
\begin{equation}
    \log  Jf(x)  - \log d = u(f(x)) - u(x).
\end{equation}
Defining the measure $d\nu = e^{u(x)}dm$, the proof that $\nu$ is $f$-invariant and mutually absolutely continuous with respect to Riemannian volume, with a strictly positive $C^1(M)$ density, follows exactly the same computational steps detailed in \textbf{Step 5} above.
\end{proof}

\section{Proof of Theorem \ref{TeoA}}

\begin{proof}
	
We divided the proof into five steps. \vspace{0.1cm}

\textbf{Step 1. Existence of an inverse SRB measure $\nu$ for $f$.}\vspace{0.1cm}
	
By virtue of Proposition \ref{PropAA}, the Anosov property of $f$ guarantees that the non-wandering set coincides with the closure of its periodic points, \mbox{$\Omega(f)=\overline{Per(f)}$}, establishing $f$ as an Axiom A endomorphism. According to Remark \ref{rem:srb_degree}, $f$ must possess a repelling basic set $\Omega_k(f)$. The complete invariance of $f$ directly implies that the repeller $\Omega_k(f)$ is also completely invariant. Consequently, by applying Remark \ref{rem:srb_degree} along with the developments in the preceding section, by \cite{M-2010,PDL2}, we ensure the existence of a unique inverse SRB measure $\nu$ supported on $\Omega_k(f)$ which satisfies the entropy formula:

\begin{equation}\label{Eq1}
h_{\nu}(f) = \log d -\int_{M} \Lambda^s_f(x) \, d\nu.
\end{equation}\vspace{0.1cm}

\textbf{Step 2. The inverse SRB measure $\nu$ is also an SRB measure.}
\vspace{0.2cm}

Let $R$ denote the set of points in $M$ that are simultaneously recurrent and Lyapunov regular under $f$, so that $\nu(R)=1$. By combining the Anosov Closing Lemma with our periodic Jacobian assumption, namely $Jf^n(p)=d^n$ for all periodic points $p \in M$ of period $n$, we aim to establish that for any $x \in R$:
$$
\Lambda^u_f(x) + \Lambda^s_f(x)=\log d.
$$

To prove this relation, fix $\varepsilon>0$ and take $x\in R$. The $C^1$ regularity of $f$ provides a real parameter $\delta>0$ such that $d(x,y)\leq \delta$ ensures the ratio of their Jacobians is bounded by:
$$
1-\varepsilon < \frac{Jf(x)}{Jf(y)} < 1+\varepsilon.
$$

The recurrence of $x$ implies we can find a sequence of times $n_k \to \infty$ for which $d(x,f^{n_k}(x))\leq \delta'$ for an appropriate choice of $\delta'>0$. This allows us to use the Closing Lemma to find a periodic point $p_k \in M$ satisfying $f^{n_k}(p_k)=p_k$ whose orbit $\delta$-shadows the finite pseudo-orbit of $x$. Specifically, we have $d(f^i(x),f^i(p_k))\leq\delta$ for every $i=0,1,\ldots, n_k-1$.

Since the periodic Jacobian at $p_k$ is given by $Jf^{n_k}(p_k)=d^{n_k}$, we can write the ratio of the accumulated Jacobians along these orbits as:
$$
\frac{Jf^{n_k}(x)}{d^{n_k}}=\frac{Jf^{n_k}(x)}{Jf^{n_k}(p_k)}
=\prod_{j=1}^{n_k-1} \frac{Jf(f^j(x))}{Jf(f^j(p_k))}.
$$

Applying our previous bound to each term in the product yields:
$$
(1-\varepsilon)^{n_k}<\frac{Jf^{n_k}(x)}{Jf^{n_k}(p_k)}<(1+\varepsilon)^{n_k},
$$
which immediately leads to the estimate:
$$
(1-\varepsilon)^{n_k}d^{n_k}<Jf^{n_k}(x)<(1+\varepsilon)^{n_k}d^{n_k}.
$$

By taking the logarithm, dividing by $n_k$, and computing the limit as $k\to \infty$, we obtain:
$$
\log(1-\varepsilon)+\log d\leq \Lambda^u_f(x) + \Lambda^s_f(x) \leq \log(1+\varepsilon)+\log d.
$$

Since $\varepsilon > 0$ can be chosen arbitrarily small, letting $\varepsilon \to 0$ yields the identity $\Lambda^u_f(x) + \Lambda^s_f(x)=\log d$. Substituting this relation back into the entropy formula \eqref{Eq1}, we deduce that:
$$
h_{\nu}(f) =\int_{M} \Lambda^u_f(x) \, d\nu.
$$

By \cite{QZ}, we conclude that $\nu$ is indeed an SRB measure, possessing a positive, absolutely continuous density along the unstable manifolds.
\vspace{0.3cm}

 \textbf{Step 3. The transitive set $\Omega_k(f)$ has nonempty interior.}\vspace{0.1cm}

Since $\nu$ is a SRB measure, there exists a measurable partition $\eta$ of $M^f$ subordinate to the unstable manifolds $W^u$ of $f$ such that, for the disintegration of $\bar{\nu}=\pi_{\ast}^{-1}\nu$	 with respect to $\eta$, denoted by $\{\bar{\nu}_{\eta(\bar{x})}\}_{\bar{x}\in M^f}$, the projected conditional measures  $\pi_{\ast}\bar{\nu}_{\eta(\bar{x})}$ are equivalent to $m_{\pi(\eta(\bar{x}))}$ for $\bar{\nu}$-a.e. $\bar{x}$, where $m_{\pi(\eta(\bar{x}))}$ denotes the Lebesgue measure on $W^u(\bar{x})$ restricted to $\pi(\eta(\bar{x}))$. Hence,
$$
1=\bar{\nu}(\Omega_k^f)=\int_{M^f}\bar{\nu}_{\eta(\bar{x})}(\eta(\bar{x})\cap \Omega_k^f)d\bar{\nu}.
$$

Take $\bar{x}\in M^f$ with $\bar{\nu}_{\eta(\bar{x})}(\eta(\bar{x})\cap \Omega_k^f)=1$, then $\pi_{\ast}\bar{\nu}_{\eta(\bar{x})}(\pi(\eta(\bar{x}))\cap \Omega_k(f))=1$, by definition of the partition $\eta$, there is an unstable disc $D^u(y)\subset \pi(\eta(\bar{x}))$ such that
$\pi_{\ast}\bar{\nu}_{\eta(\bar{x})}(D^u(y))=\pi_{\ast}\bar{\nu}_{\eta(\bar{x})}(D^u(y)\cap \Omega_k(f))$. Since $\pi_{\ast}\bar{\nu}_{\eta(\bar{x})}$ is equivalent to  $m_{\pi(\eta(\bar{x}))}$, it follows that $m_{\pi(\eta(\bar{x}))}(D^u(y))=m_{\pi(\eta(\bar{x}))}(D^u(y)\cap \Omega_k(f))$. Consequently,
$\Omega_k(f)$ is dense in $D^u(y)$. Since $\Omega_k(f)$ is closed, it follows that
$D^u(y)\subset\Omega_k(f)$.

Since $\Omega_k(f)$ is a repeller, it is saturated by stable manifolds, that is \mbox{$W^s(\Omega_k)=\Omega_k$}. Therefore, $\displaystyle\bigcup_{x\in D^u(y)}W^s(x)\subset\Omega_k(f)$. Hence, there is an open ball
$$B_r(y)\subset \displaystyle\bigcup_{x\in D^u(y)}W^s(x)\subset\Omega_k(f).$$ Therefore, $\Omega_k(f)$ is a hyperbolic set with nonempty interior.

\vspace{0.2cm}

 \textbf{Step 4. $\Omega_k(f)=M$ .}\vspace{0.1cm}

We denote $U=int(\Omega_k(f))$. Since $\Omega(f)$ is completely invariant, each basic set is also completely invariant. In particular $\Omega_k(f)$ is completely invariant, and it is straightforward to verify that both sets $U$ and $\overline{U}$ are completely invariant too.

Since $f|_{\Omega_k}$ is transitive, by Theorem \ref{Teo-transitive}, it follows that $\bar{f}|_{\Omega_k^f}$ is also transitive. Let $\bar{x}\in\Omega_k^f$ be a point with dense orbit and such that $\pi(\bar{x})=x_0\in U$. We claim that $\Omega_k^f$ is $u$-saturated, that is,
$
W^u(\bar{y})\subset \Omega_k(f)
\quad \text{for every } \bar{y}\in \Omega_k^f.
$
To prove this, let $\bar{y}\in \Omega_k^f$ and consider a sequence $\bar{f}^{n_i}(\bar{x})$ converging to $\bar{y}$. Since $x_0\in U,$ there exists $\varepsilon>0$ such that $
W^u_\varepsilon(\bar{x})\subset U.
$ By invariance, for every $n_i$,
$f^{n_i}(W^u_\varepsilon(\bar{x}))\subset U.$ Moreover, since $f$ is expanding along unstable manifolds, the diameter of $f^{n_i}(W^u_\varepsilon(\bar{x}))$ tends to infinity as $n_i\to\infty.$ Therefore, fixed any $r> 0,$ by Theorem \ref{PRZ-1}, the sequence $f^{n_i}(W^u_\varepsilon(\bar{x}))$ converges to $W^u_{r}(\bar{y}).$ Since $\Omega_k(f)$ is closed and $f^{n_i}(W^u_\varepsilon(\bar{x}))\subset U\subset \Omega_k(f)$ for all $n_i$, we conclude that
$W^u_r(\bar{y})\subset \Omega_k(f),$ for any arbitrary $r> 0.$ Thus, $\Omega_k^f$ is $u$-saturated. Since $\Omega_k(f)$ is a repeller, it is also $s$-saturated. Consequently, $\Omega_k(f)$ is open. As it is also closed and $M$ is connected, it follows that
$\Omega_k(f)=M$. Hence, $f$ is transitive.

\vspace{0.1cm}

\textbf{Step 5. $f$ preserves a (unique) normalized $C^1$ volume measure.} The existence of such a measure $\mu$ follows straightforwardly from Corollary \ref{Cor1}.
It remains to prove that $\mu$ is unique. Assume that there exist two $C^1$ volume measures, $\mu$ and $\nu$, which are $f$-invariant. Since they are smooth measures, both $\mu$ and $\nu$ are SRB and inverse measures. By Step 4, $f$ has a unique repeller, and therefore, $f$ admits a unique inverse SRB measure. We conclude that $\mu=\nu$.

\end{proof}

%###################################################################

\section{Proof of Theorem \ref{Teo-exem}}

\begin{proof}
	Consider the map $g(x)=2x(mod\, 1)$ on $\mathbb{S}^1$ and the fixed point $p=0\in \mathbb{S}^1$. For small $\varepsilon>0$, the function  $g_1(x)=2x -\varepsilon\sin(2\pi x)(mod\, 1)$ is $C^1-$close to $g$ with $Jg_1(0)=g_1'(0)\neq 2$ and $\deg(g_1)=2$. Since $g_1$ is expanding, it preserves a volume $\mu$ (see \cite{KS95}).
	Now consider $A:\T^2\to \T^2$, the Anosov diffeomorphism induced by the matrix
	$	\left[\begin{array}{cc}
		2 & 1 \\
		1 & 1
	\end{array}
	\right]$. Therefore, the map $f=g_1\times A$ is a $C^{\infty}$ transitive Anosov endomorphism of $\T^3$ with $\deg(f)=2$, preserving the volume $\nu=\mu\times m$. However, for the fixed point $0\in\T^3$, we have $Jf(0)\neq 2$.	
	
	We claim that $\nu$ is an SRB measure for $f$, but not an inverse SRB measure. Let $\lambda_A^u$ and $\lambda_A^s$ denote the unstable and stable Lyapunov exponents of $A$, respectively, with respect to $m$, and let $\lambda^u(g_1,\mu)$ denote the Lyapunov exponent of $g_1$ with respect to $\mu$.
	
	First, observe that
	$$
	h_\mu(g_1)=\lambda^u(g_1,\mu)<\log 2=h_{\mathrm{top}}(g_1).
	$$
	Indeed, if equality held, then $\mu$ would be the measure of maximal entropy for $g_1$. Consequently, from \cite{SS85}, the conjugacy between $g$ and $g_1$ would be $C^1$, implying that every periodic point of $g_1$ has Lyapunov exponent equal to $\log 2$.
	
	Now, suppose, for a contradiction, that $\nu$ is an inverse SRB measure for $f$. The entropy formula for inverse SRB measures would then yield
	$$
	\log 2-\lambda_A^s
	=h_\nu(f)
	=\lambda_A^u+\lambda^u(g_1,\mu).
	$$
	Since $A$ preserves area, its Lyapunov exponents satisfy
	$$
	\lambda_A^u+\lambda_A^s=0.
	$$
	It follows that
	$$
	\lambda^u(g_1,\mu)=\log 2,
	$$
	contradicting the strict inequality established above. Hence, $\nu$ is not an inverse SRB measure.
	
	On the other hand, $\nu$ is a volume measure invariant under $f$, and the Pesin entropy formula implies that $\nu$ is an SRB measure. Therefore, $\nu$ is SRB but not inverse SRB.
\end{proof}

It is important to note that, for the map $f = g_1 \times A$, the stable foliation is absolutely continuous, as in the previous setting. Thus, by \cite{PDL2}, the measure $\nu$ satisfies
$$
h_{\nu}(f) = F_{\nu}(f) - \lambda^s_A,
$$
where
$$
F_{\nu}(f) = H_{\nu}(\xi \mid f^{-1}\xi),
$$
with $\xi$ being a measurable partition subordinate to the stable foliation. Consequently, we obtain
$$
F_{\nu}(f) = h_{\nu}(g_1).
$$

On the other hand, let $\mu_{\max}$ denote the measure of maximal entropy of $g_1$. Then
$$
\begin{aligned}
h_{\mu_{\max} \times m}(f)
&= h_{\mu_{\max}}(g_1) + h_m(A) \\
&= \log(2) + \lambda^u_A \\
&= \log(2) - \lambda^s_A.
\end{aligned}
$$
Therefore, by \cite{M-2010},  $\mu_{\max} \times m$ is the inverse SRB measure for $f$.

\section{Proof of Theorem \ref{TeoC}}

From \cite{AAS-2003} or \cite{C-2003} we can state the following lemma which will be useful to prove Theorem \ref{TeoC}.

\begin{lemma}[\cite{AAS-2003}, \cite{C-2003}]\label{LemC}
Let $\mathcal{M}$ be the set of $f$-invariant measures and $\varphi$ be a continuous function on $M$. If $\alpha<\int\varphi d\mu<\lambda,$ $ \forall \mu\in \mathcal{M}$, then there exists $N$ such that for every $n\geq N:$
$$
\displaystyle\alpha<\int\frac{1}{n}\sum_{i=0}^{n-1}\varphi(f^i(x)) d\mu<\lambda.
$$
for all $x\in M.$ 	
\end{lemma}

\begin{proof}[Proof of Theorem \ref{TeoC}]
Let $\Lambda^u_f(p)$ and $\Lambda^s_f(p)$ denote, respectively, the sum of all positive and negative Lyapunov exponents of $f$ at a periodic point $p$, each weighted by its multiplicity $\dim E_i(p)$ in the Oseledets splitting. Let $R$ be the set of all simultaneously regular and recurrent points of $f$, so $\mu(R)=1$ for every $f$-invariant Borel measure $\mu.$ As in the proof of Theorem \ref{TeoA}, using the Anosov Closing Lemma to get
$$
\Lambda^u_f(x)+ \Lambda^s_f(x)=\Lambda^u+ \Lambda^s,
$$ 	
for all $x\in R$. By Lemma \ref{LemC}, the limit

\begin{equation*}\label{eq4}
\displaystyle\lim_{n\to\infty}\dfrac{1}{n}\log Jf^n(x)=\Lambda^u+ \Lambda^s=\lambda,	
\end{equation*}
is uniform in $x\in M.$ Then, given $\varepsilon$, there is $N>0$ such that for $n\geq N$ and all $x\in M$ we have
\begin{equation*}
	\lambda-\varepsilon\leq \dfrac{1}{n}\log Jf^n(x) \leq \lambda+\varepsilon
\end{equation*}

\begin{equation}\label{eq51}
	e^{n(\lambda-\varepsilon)}\leq Jf^n(x) \leq e^{n(\lambda+\varepsilon)}.
\end{equation}

 Since $\deg(f)=d$, by degree theory (see for instance \cite{KH95}, pp. 310-312) we have

\begin{equation*}
\int_M Jf(x)d m=d\,\,\,\,\, \Rightarrow\,\,\,\,\, \int_M Jf^n(x)d m=d^n.	
\end{equation*}
So
\begin{equation}\label{eq5}
\frac{1}{n}\log\int_M Jf^n(x)d m=\log d.		
\end{equation}

Using the equalities \ref{eq51} and \ref{eq5} for $n\geq N$ we have
$$
	\lambda-\varepsilon\leq \log d \leq \lambda+\varepsilon,
$$
since $\varepsilon$ is arbitrary, $\log d=\lambda=\Lambda^u(x)+ \Lambda^s(x)$ for all $x\in M$. Consequently, $Jf^n(p)=d^n$ for all $p\in M$ with $f^n(p)=p$ and by Theorem \ref{TeoA}, $f$ is transitive.
\end{proof}

We remark that the equation $\log d=\lambda=\Lambda^u(x)+ \Lambda^s(x)$, stated for all $x\in M$ at the end of the proof, does not mean that every point of $M$ is Lyapunov regular; it only means that the limit $\displaystyle\lim_{n\to\infty}\dfrac{1}{n}\log Jf^n(x)$ converges to the constant $\lambda$.

\section{Appendix -- Closing Lemma with Exponential Control}\label{appendix}

The proof of the Hölder version of the Livšic Theorem for
transitive Anosov endomorphisms requires a closing lemma with
exponential control, analogous to the one used in the
invertible setting. Although the shadowing property guarantees
that periodic pseudo-orbits are shadowed by periodic orbits,
the exponential estimates needed for the Hölder regularity
argument require additional geometric considerations.

In this appendix, we establish such estimates by exploiting
the local product structure of Anosov endomorphisms and the
uniform hyperbolicity of their lifts.

By a result of Sakai \cite{Sak}, transitive Anosov endomorphisms satisfy the shadowing property. Moreover, every periodic pseudo-orbit of period $m$ is shadowed by a periodic orbit of the same period.

\begin{proposition}\label{closing}\cite{Sak}
Let $f:M\rightarrow M$ be an Anosov endomorphism. Given $\varepsilon>0$, there exists $0<\delta<\varepsilon$ such that, whenever
\[
x,f(x),\ldots,f^m(x)
\]
is a closed $\delta$-pseudo-orbit, there exists a periodic point $z\in M$ satisfying $f^m(z)=z$ and
\[
d(f^k(x),f^k(z))<\varepsilon,
\qquad k=0,\ldots,m-1.
\]
\end{proposition}

A key ingredient in establishing exponential control is the local product structure, as described in \cite{QXZ}.

\begin{proposition}[Theorem IV.2.3 of \cite{QXZ}] \label{transversal}
Let $f:M\to M$ be an Anosov endomorphism. For sufficiently small $\varepsilon>0$, there exists $0<\delta<\varepsilon$ such that, whenever $\bar{y}\in M^f$ and $x_0\in M$ satisfy $d(y,x_0)<\delta$, the unstable manifold $W_\varepsilon^u(\bar{y})$ intersects the stable manifold $W_\varepsilon^s(x_0)$ transversally at a unique point $z_{x_0,\bar{y}}$.
\end{proposition}

An immediate consequence of this result is that every lift $\tilde{f}:\widetilde{M}\to \widetilde{M}$ of an Anosov endomorphism has a uniform local product structure. Indeed, since $W_\varepsilon^{u}(\bar{y})$ varies continuously with $\bar{y}\in M^f$ and $M^f$ is compact, the constant $\delta$ depends only on $\varepsilon$.

\begin{proposition}[Closing Lemma with Exponential Control]\label{exp-closing}
	Let $f:M\to M$ be an Anosov endomorphism. There exist constants $L>0$, $0<\lambda<1$, and a sufficiently small $\varepsilon_0>0$ such that, letting $\delta_0=\delta(\varepsilon_0)$ be as in Proposition~\ref{closing}, if $x,z\in M$ are as in Proposition~\ref{closing} for some integer $m\geq 1$, then
	\[
	d(f^k(x),f^k(z))\leq L\max\{\lambda^k,\lambda^{m-k}\}\,\varepsilon_0
	\]
	for every $k=0,\ldots,m-1$.
\end{proposition}

\begin{proof}
Consider a lift $\tilde{f}:\widetilde{M}\to \widetilde{M}$. Fix a sufficiently small $\varepsilon>0$ and set $\varepsilon_0=\delta(\varepsilon)$, as given by Proposition \ref{transversal}. Next, let $\delta_0=\delta(\varepsilon_0)$ be the constant provided by Proposition \ref{closing}.

Suppose that $x$ and $z$ are as in Proposition \ref{closing}, with the corresponding constants $\varepsilon_0$ and $\delta_0$ chosen as above. Consider lifts $\tilde{x},\tilde{z}\in\tilde{M}$ such that
\[
d(\tilde{f}^k(\tilde{x}),\tilde{f}^k(\tilde{z}))<\varepsilon_0,
\qquad k=0,\ldots,m-1.
\]

The estimates below follow from the uniform hyperbolicity
of $f$ and the local product structure. Since $\varepsilon$
is sufficiently small, the local stable and unstable manifolds
can be regarded, in suitable local coordinates, as small
perturbations of their tangent spaces. Consequently, the
contraction and expansion estimates along these manifolds
yield the desired exponential bounds.

Let $0<\lambda<1$ be a hyperbolicity constant for $f$. By the local product structure in $\tilde{M}$, we define
\[
\tilde{z}_j^u :=
W_\varepsilon^s(\tilde{f}^j(\tilde{z}))
\cap
W_\varepsilon^u(\tilde{f}^j(\tilde{x})),
\]
and
\[
\tilde{z}_j^s :=
W_\varepsilon^u(\tilde{f}^j(\tilde{z}))
\cap
W_\varepsilon^s(\tilde{f}^j(\tilde{x})),
\]
for $j=0,\ldots,m.$

Fix $0\leq k\leq m-1$. By considering forward iterates of $\tilde{f}$ and using contraction along stable manifolds, after adjusting the constants, we obtain
\[
d(\tilde{z}_k^s,\tilde{f}^k(\tilde{x}))
\leq C\lambda^k d(\tilde{z}_0^s,\tilde{x})
\leq C_1\lambda^k d(\tilde{z},\tilde{x}),
\]
and
\[
d(\tilde{z}_k^u,\tilde{f}^k(\tilde{z}))
\leq C\lambda^k d(\tilde{z}_0^u,\tilde{z})
\leq C_2\lambda^k d(\tilde{z},\tilde{x}).
\]

Similarly, by considering backward iterates of $\tilde{f}$ and using contraction along unstable manifolds, we obtain
\[
d(\tilde{f}^k(\tilde{z}),\tilde{z}_k^s)
\leq C\lambda^{m-k}
d(\tilde{z},\tilde{z}_{m}^s)
\leq C_3\lambda^{m-k}
d(\tilde{z},\tilde{f}^{m}(\tilde{x})),
\]
and
\[
d(\tilde{z}_k^u,\tilde{f}^k(\tilde{x}))
\leq C\lambda^{m-k}
d(\tilde{z}_{m}^u,\tilde{f}^{m}(\tilde{x}))
\leq C_4\lambda^{m-k}
d(\tilde{z},\tilde{f}^{m}(\tilde{x})).
\]

Finally, since the natural projection $\hat{p}:\tilde{M}\to M$ is a local isometry, we conclude that
\[
d(f^k(x),f^k(z))
\leq L\max\{\lambda^k,\lambda^{m-k}\}\,\varepsilon_0,
\]
where
\[
L=2\max\{C,C_1,C_2,C_3,C_4\}.
\]	
\end{proof}

This establishes the desired exponential control. Consequently, the proof of the Hölder version of the Livšic Theorem can be carried out in the same way as in the case of transitive Anosov diffeomorphisms.

\end{document}